\documentclass{amsart}

\usepackage{amsmath}
\usepackage{amssymb}
\usepackage{epsfig}  		
\usepackage{bbm}
\usepackage{hyperref}

\newtheorem{thm}{Theorem}[section]
\newtheorem{prop}[thm]{Proposition}
\newtheorem{lem}[thm]{Lemma}
\newtheorem{cor}[thm]{Corollary}

\theoremstyle{definition}
\newtheorem{definition}[thm]{Definition}
\newtheorem{example}[thm]{Example}

\theoremstyle{remark}

\newtheorem{remark}[thm]{Remark}

\numberwithin{equation}{section}

\newcommand{\CC}{\mathbbm{C}}
\newcommand{\RR}{\mathbbm{R}}

\newcommand{\id}{\mathrm{id}}

\newcommand{\Iso}{\mathrm{Iso}}
\newcommand{\Aff}{\mathrm{Aff}}

\newcommand{\Diff}{\mathrm{Diff}}
\newcommand{\hol}{\mathrm{hol}}

\DeclareMathOperator{\im}{\mathrm{im}}

\newcommand{\frg}{\mathfrak{g}}
\newcommand{\fra}{\mathfrak{a}}

\newcommand{\frz}{\mathfrak{z}}
\newcommand{\sso}{\mathfrak{so}}
\newcommand{\iso}{\mathfrak{iso}}

\newcommand{\ra}{\rightarrow}

\begin{document}


\title[Flat Pseudo-Riemannian Homogeneous Spaces]{Flat Pseudo-Riemannian Homogeneous Spaces with Non-Abelian Holonomy Group}

\author[Baues]{Oliver Baues}
\address{Department of Mathematics, 
Institute for Algebra and Geometry\\
Karlsruhe Institute for Technology, 76131 Karlsruhe, Gemany}
\email{baues@kit.edu}

\author[Globke]{Wolfgang Globke}
\address{Department of Mathematics, 
Institute for Algebra and Geometry\\
Karlsruhe Institute for Technology, 76131 Karlsruhe, Germany}
\email{globke@math.uni-karlsruhe.de}

\begin{abstract} 
We construct homogeneous flat 
pseudo-Riemannian manifolds with 
non-abelian fundamental group. 
In the compact case, all homogeneous 
flat pseudo-Riemannian manifolds are
complete and have abelian linear holonomy group.
To the contrary, we show that there do
exist non-compact and non-complete examples, where the linear
holonomy is non-abelian, starting in dimensions $\geq 8$, which is the lowest possible dimension. We also construct 
a complete flat pseudo-Riemannian homogeneous manifold
of dimension 14 with non-abelian linear holonomy.
Furthermore, we derive a criterion for the properness of the action
of an affine transformation group with 
transitive centralizer.
\end{abstract}

\maketitle
\tableofcontents
\section{Introduction}
A  flat  pseudo-Riemannian 
manifold $M$ is called homogeneous 
if its isometry group acts transitively. As examples 
show \cite{DuIh,DuIh_1}, non-compact flat pseudo-Riemann\-ian homogeneous manifolds manifolds are not necessarily complete. 
The study of  complete flat  homogeneous  
pseudo-Riemannian manifolds 
was pioneered by Wolf
in a series of papers  \cite{Wolf_1, Wolf_2, Wolf_3}. 
Such manifolds are 
isometric to a manifold of the form $\RR^{r,s}/\, \Gamma$, for some subgroup $\Gamma\leq {\rm Iso}(\RR^{r,s})$.
Homogeneity  implies that the centralizer of $\Gamma$ in 
$\Iso(\RR^{r,s})$ acts transitively on $\RR^{r,s}$. 
One basic observation 
is that  in this situation the group $\Gamma$  is nilpotent of nilpotency  class at most two. 
This fact also holds for the holonomy group $\Gamma$ of a 
non-complete homogeneous pseudo-Riemannian manifold.
Apparently, it was believed for some time that 
$\Gamma$, or, which is considerably weaker, the linear part of $\Gamma$ should be abelian. However, as observed 
in 
\cite{Baues} non-abelian fundamental 
groups $\Gamma$ appear 
for compact complete flat homogeneous 
pseudo-Riemannian manifolds. \\
%

In this note, we present some additional new results on the structure of flat pseudo-Riemannian homogeneous manifolds. 
Although non-abelian fundamental groups $\Gamma$ do appear,
in the compact case the linear holonomy is always  abelian. 
In addition, we show 
that every homogeneous  flat pseudo-Riemannian manifold of dimension less than eight 
has abelian linear holonomy. 

As one our main results, we give examples of homogeneous manifolds with non-abelian linear holonomy group. We construct an eight-dimensional non-complete manifold $U / \, \Gamma_1$, where $U$ is an open domain
in $\RR^{4,4}$, and a fourteen-dimensional 
complete manifold
$\RR^{7,7}/ \, \Gamma_2$, both with non-abelian
linear holonomy. The groups
$\Gamma_{1} \leq {\rm Iso}(\RR^{4,4})$ and 
$\Gamma_{2} \leq {\rm Iso}(\RR^{7,7})$ are isomorphic to the integral Heisenberg group on two generators and map injectively to their linear parts. 
These manifolds give the first examples of flat pseudo-Riemannian homogeneous manifolds with non-abelian linear holonomy group. 
%
%
%
\section{Preliminaries}
Here, $\RR^{r,s}$ denotes
$\RR^{r+s}$ endowed with a scalar product
$\langle\cdot,\cdot\rangle$ of signature $(r,s)$, 
and $\Iso(\RR^{r,s})$ its group of isometries.
Affine maps of $\RR^{r+s}$ are written 
as $\gamma= (I+A,v)$, where $I+A$ is the linear part ($I$ the
identity matrix), and $v$ the translation part.

The groups $\Gamma\subset\Iso(\RR^{r,s})$ with 
transitive centralizer in $\Iso(\RR^{r,s})$ were 
studied first in \cite{Wolf_1}. 
We sum up some of the results for later reference.
Note that all of the following holds also if the 
centralizer of $\Gamma$ is only required to have
an open orbit in  $\RR^{r,s}$ (compare  \cite[Proposition 3.10]{Baues} or
\cite[Lemma 4.1]{DuIh_2}).

\begin{lem}\label{lem_wolf1}
$\Gamma$ consists of affine transformations $\gamma=(I+A,v)$, where $A^2=0$,
$v\perp\im A$ and $\im A$ is totally isotropic.
\end{lem}

\begin{lem}\label{lem_wolf2}
For $\gamma_i=(I+A_i,v_i)\in\Gamma$, $i=1,2,3$, we have
$A_1 A_2 v_1=0=A_2A_1v_2$,
$A_1 A_2 A_3=0$
and
$[\gamma_1,\gamma_2] = (I+2A_1 A_2, 2A_1 v_2)$.
\end{lem}

\begin{lem}\label{lem_wolf3}
If $\gamma=(I+A,v)\in\Gamma$, then $\langle Ax,y\rangle=-\langle x,Ay\rangle$,
$\im A=(\ker A)^\bot$, $\ker A=(\im A)^\bot$ and $Av=0$.
\end{lem}

\begin{thm}\label{thm_wolf1}
$\Gamma$ is 2-step nilpotent (meaning $[\Gamma,[\Gamma,\Gamma]]=\{\id\}$).
\end{thm}

For $\gamma=(I+A,v)\in\Gamma$, set $\hol(\gamma)=I+A$ (the linear component of
$\gamma$). We write $A=\log(\hol(\gamma))$.

\begin{definition}
The \emph{linear holonomy group} of $\Gamma$ is
$\hol(\Gamma)=\{\hol(\gamma) \mid \gamma\in \Gamma\}$.
\end{definition}

In the latest edition of the book \cite{Wolf_buch}, a characterization of those $\Gamma$
with abelian linear holonomy is given:

\begin{prop}\label{prop_wolf1}
The following are equivalent:
\begin{enumerate}
\item
$\hol(\Gamma)$ is abelian.
\item
If $(I+A_1,v_1), (I+A_2,v_2)\in\Gamma$, then $A_1 A_2=0$.
\item
The space $U_\Gamma=\sum_{\gamma\in\Gamma}\im A$ is totally isotropic.
\end{enumerate}
\end{prop}

The proof of Proposition \ref{prop_wolf1} uses only the lemmata above. 
Thus, the  following structure theorem for groups $\Gamma$, such
that the centralizer of $\Gamma$ has an open orbit in  $\RR^{r,s}$,
and $\Gamma$  with abelian linear holonomy, holds:

\begin{thm}\label{thm_wolf2}
If $\hol(\Gamma)$ is abelian, then for every Witt basis with
respect to $U_\Gamma$ (see section \ref{sec_nonabelian}) and
 $(I+A,v)\in\Gamma$, the matrix $A$ is 
of the form
\begin{equation}
A=
\begin{pmatrix}
0 & 0 & C \\
0 & 0 & 0 \\
0 & 0 & 0
\end{pmatrix},
\label{eq_wolf4}
\end{equation}
where $C$ is a skew-symmetric $k\times k$-matrix ($k=\dim U_\Gamma$).
\end{thm}

In section \ref{sec_compact}, we show that for compact $M$ the holonomy group $\hol(\Gamma)$ is always abelian, and we present a more refined classification and structure theorem for the groups $\Gamma$ in the  compact case.
We give a modification for the structure theorem, Theorem \ref{thm_wolf2},  which holds for arbitrary $\Gamma$, in section \ref{sec_nonabelian}.
In section \ref{sec_dimbound}, we show that  groups $\Gamma$ such that
$\hol(\Gamma)$ is not abelian exist only for
dimensions $\geq 8$, and in section \ref{sec_examples} we give an example
of such a group. This group does not act freely and therefore cannot be
the fundamental  group of a complete flat homogeneous pseudo-Riemannian
manifold $M$. But it gives rise to a non-complete example $M$. We also 
present an example of a group $\Gamma$ which acts freely on $\RR^{7,7}$ and has a transitive centralizer. This group gives rise
to a complete 14-dimensional homogeneous flat pseudo-Riemannian manifold with non-abelian linear holonomy group. 
To show that the groups $\Gamma$ involved act properly 
we derive in section \ref{sec_proper} a criterion which shows
that a discrete unipotent group acting freely on $\RR^n$, and whose centralizer has an open orbit, acts properly on $\RR^n$.

\section{Compact Flat Pseudo-Riemannian Spaces}
\label{sec_compact}

In this section, let $M$ be a compact flat homogeneous pseudo-Riemannian mani\-fold. 
By \cite{Marsden} (see \cite[Corollary 4.5]{Baues} for an alternative proof), $M$ must be
complete. Therefore, $M=\RR^{r,s}/\,\Gamma$, for some group $\Gamma \leq \Iso(\RR^{r,s})$ which acts properly discontinuously and freely on $\RR^{r,s}$. 
Let $G$ be the centralizer of $\Gamma$ in $\Iso(\RR^{r,s})$. Since $M$ is homogeneous and compact, then, as follows from \cite[Theorem 4.6]{Baues}, $G$ is a nilpotent Lie group which acts simply transitively on 
$\RR^{r,s}$ by isometries. 
Let $x_0\in\RR^{r,s}$ be a fixed basepoint. There is a unique left invariant  pseudo-Riemannian metric $\langle\cdot,\cdot\rangle_G$ on $G$ such that  the orbit map $o:G\rightarrow\RR^{r,s}$, $g\mapsto g \cdot x_0$, is an isometry. Moreover,  the metric $\langle\cdot,\cdot\rangle_G$ is biinvariant,
see \cite[Theorem 4.6]{Baues}. The map $o$ induces an
isometry $G/\/ \tilde \Gamma \ra M$, 
where $\tilde \Gamma$ is a lattice subgroup of $G$,
which is isomorphic to $\Gamma$, and $G/\/\tilde \Gamma$
inherits the pseudo-Riemannian structure from $(G, \langle\cdot,\cdot\rangle_G)$.
It also follows that $G$ is (at most) two-step nilpotent (see \cite[Lemma 4.8]{Baues}).
Such manifolds $G/\/\tilde \Gamma$ 
necessarily have abelian linear holonomy 
group:

\begin{thm}\label{thm_compact1}
Let $G$ be a Lie group with a biinvariant flat pseudo-Riemannian metric $\langle\cdot,\cdot\rangle_G$, and $\tilde \Gamma \leq G$ be a
lattice. Then the compact flat pseudo-Riemannian homogeneous manifold  $G / \/ \tilde \Gamma$ has abelian linear holonomy.
\end{thm}
\begin{proof}
Let $\rho: G \ra \Iso(\RR^{r,s})$ be the development representation
of the right-multiplication of $G$ and put $\Gamma = \rho(\tilde \Gamma)$. Then, as above, there is an orbit map 
$o: G \ra \RR^{r,s}$, which is an isometry and satisfies
$o(g \gamma) = \rho(\gamma) o(g)$. (cf.\ \cite[Proposition 5.2]{Baues}.) This map induces an isometry $G/ \/ \tilde \Gamma \ra \RR^{r,s}\! /\, \Gamma$. 

Let $\frg$ denote the Lie algebra of $G$. By \cite[Proposition 3.3, Lemma 5.10]{Baues}, the differential of $\rho$ at the identity is equivalent to 
the affine representation $X \mapsto (\frac{1}{2}\mathrm{ad}(X),X)$ of $\frg$ on the vector space of $\frg$. In particular, the linear 
part of the differential of $\rho$ is equivalent to the adjoint representation $\mathrm{ad}$ of $\frg$. Since $\frg$ is two-step nilpotent, the
adjoint representation $\mathrm{ad}$ 
has abelian image. It follows that the
linear part of $\rho(G)$ is abelian. Since $\Gamma \leq \rho(G)$, 
this implies that $\Gamma$ has abelian linear part. 
\end{proof}

\begin{remark} Let $\langle\cdot,\cdot\rangle_{\frg}$ denote the inner product induced on $\frg$ by $\langle\cdot,\cdot\rangle_G$. Biinvariance of $\langle\cdot,\cdot\rangle_G$ is equivalent to
\begin{equation}
\langle [X,Y],Z \rangle_{\frg} = - \langle Y, [X,Z] \rangle_{\frg}
\label{eq_compact1}
\end{equation}
Identify $\frg\cong\RR^{r,s}$ via the differential of $o$.
Then $\gamma=(I+A,v)=(I+\frac{1}{2}\mathrm{ad}(X),X)$,
where $X=\log(\gamma)\in\frg$. Therefore, 
\begin{equation}
U_\Gamma=\sum_{\gamma\in\Gamma}\im A
=\sum_{X\in\log(\Gamma)}\im\mathrm{ad}(X)
=[\frg,\frg],
\label{eq_compact2}
\end{equation}
equals the commutator subalgebra of $\frg$. 
(The last equality holds because the ele\-ments 
in $\log(\Gamma)$ generate $\frg$, since $\Gamma$
is a lattice in $G$.) 
Using biinvariance and 2-step nilpotency,
it is easy to see that the space $U_\Gamma = [\frg,\frg]$ 
is totally isotropic.
By Theorem \ref{thm_wolf2}, this is equivalent to $\hol(\Gamma)$
being abelian.
\end{remark}

\begin{cor}\label{cor_compact1}
Let $M=\RR^{r,s}/\Gamma$ be a compact flat pseudo-Riemannian homogeneous manifold. Then $\hol(\Gamma)$ is abelian.
\end{cor}

To specify a biinvariant pseudo-Riemannian
metric on the Lie group $G$ it is equivalent to
construct a bi\-invariant
inner product $\langle\cdot,\cdot\rangle_\frg$ on 
the Lie algebra of $\frg$. The metric is flat 
if and only if $\frg$ is two-step nilpotent. 
Below, we state a structure theorem for such pairs 
$(\frg, \langle\cdot,\cdot\rangle_\frg)$, taken from 
\cite[Theorem 5.15]{Baues}. By the above, this
yields a structure theorem for groups $\Gamma$
which are the fundamental groups of compact 
homogeneous flat pseudo-Riemanninan manifolds. \\

Recall that for  an abelian Lie algebra $\fra$ and
its dual $\fra^*$, a Lie product on the space $\fra\oplus\fra^*$
is given by
\[
[(X,X^*),(Y,Y^*)]
= ([X,Y],\ \mathrm{ad}^*(X)Y^* - \mathrm{ad}^*(Y)X^* + \omega(X,Y) ),
\]
where $\mathrm{ad}^*$ denotes the coadjoint representation,
and $\omega\in\mathrm{Z}^2(\fra,\fra^*)$ is a 2-cocycle for
the adjoint representation. 
We use the notation $\frg=\fra\oplus_{\omega}\fra^*$ for this Lie algebra.
An inner product of split signature on $\frg$ is defined by
$$ \langle (X,X^*), (Y,Y^*) \rangle_\frg = X^*(Y) + Y^*(X) , $$
and it can be shown to be biinvariant if and only if
the 3-form
$F_\omega(X_1,X_2,X_3) = \langle \omega(X_1,X_2), X_3 \rangle$
on $\fra$ is
alternating and satisfies
$$ F_\omega(X_1,[X_2,X_3],X_4) = F_\omega(X_2,X_3,[X_1,X_4]), $$
for all $X_i\in\fra$. Then $\fra\oplus_\omega\fra^*$ is a 2-step nilpotent
Lie algebra with biinvariant inner product $\langle\cdot,\cdot\rangle_\frg$.

\begin{thm}\label{thm_compact2}
Let $\frg$ be a 2-step nilpotent Lie algebra with bi\-invariant
inner product $\langle\cdot,\cdot\rangle_\frg$. Then there exists an abelian
Lie algebra $\fra$, an
alternating 3-form $F_\omega$ on $\fra$ and an abelian Lie algebra $\frz_0$
such that  $\frg$ can be written as a direct product of metric Lie
algebras
\begin{equation}
\frg = (\fra\oplus_\omega\fra^*)\oplus\frz_0.
\end{equation}
\end{thm}
\begin{proof}
$[\frg,\frg]$ is an isotropic subspace of $\frg$. Biinvariance shows 
that its ortho\-gonal complement $[\frg,\frg]^\bot$ is the center $\frz(\frg)$.
Let $\fra$ denote the isotropic subspace dual to $[\frg,\frg]$ in $\frg$
(then $[\frg,\frg]$ can be identified with the dual space $\fra^*$ of $\fra$). Finally,
let $\frz_0$ be a complement of $\fra^*$ in $\frz(\frg)$, that is
$\frz(\frg)=\fra^*\oplus\frz_0$. Then $\frz_0$ commutes with and is
orthogonal to $\fra$ and $\fra^*$.
So
$\frg = (\fra\oplus_\omega\fra^*)\oplus\frz_0$
for some 2-cocycle $\omega\in\mathrm{Z}^2(\fra,\fra^*)$.
\end{proof}

\section{Structure Theorem}
\label{sec_nonabelian}

Let $\Gamma\subset\Iso(\RR^{r,s})$ such that its centralizer in $\Iso(\RR^{r,s})$
has an open orbit in $\RR^{r,s}$.
For short, we write $\Delta$ for the center of $\Gamma$.
This group is abelian, so it satisfies the conditions of Theorem \ref{thm_wolf2}.
We set $U_\Gamma=\sum_{\gamma\in\Gamma} \im A$,
$U_{\Delta}=\sum_{\gamma\in\Delta}\im A$ and $U_0=U_{\Gamma}\cap U_{\Gamma}^\bot$, which is a totally isotropic subspace.
It follows from Lemma \ref{lem_wolf3} that
\begin{equation}
U_0 = U_\Gamma \cap U_\Gamma^\bot
= \sum_{\gamma\in\Gamma} \im A \cap \bigcap_{\gamma\in\Gamma} \ker A.
\label{eq_nonabelian1}
\end{equation}

\begin{lem}\label{lem_nonabelian1}
$U_{\Delta}\perp U_{\Gamma}$.
\end{lem}
\begin{proof}
Let $\gamma_1=(I+A_1,v_1)\in\Delta$ and $\gamma_2=(I+A_2,v_2)\in\Gamma$.
As $\gamma_1$ is central, it follows from Lemma \ref{lem_wolf2} that
$A_1 A_2=0$, so $\im A_2\subset\ker A_1=(\im A_1)^\bot$ (see Lemma \ref{lem_wolf3}).
Hence $U_{\Gamma}\perp U_{\Delta}$.
\end{proof}

It is easy to see that $U_{\Delta}\subseteq U_{\Gamma}$,
$U_{\Delta}\subseteq U_0$,
$U_{\Delta}^\bot \supseteq U_0^\bot \supseteq U_{\Gamma}^\bot\supseteq U_0 \supseteq U_{\Delta}$.

%


\begin{lem}\label{lem_nonabelian4}
$A\cdot U_{\Delta}^\bot\subseteq U_0$ for all $(I+A,v)\in\Gamma$.
\end{lem}
\begin{proof}
Let $y\in U_{\Delta}^\bot$. For all $x\in\RR^{r,s}$ and
$A,B\in\log(\hol(\Gamma))$,
\[
\langle Bx, Ay\rangle = -\langle ABx, y\rangle = 0,
\]
by Lemma \ref{lem_wolf3} and because $AB$ is central.
Hence $Ay\perp U_{\Gamma}$, that is, $Ay\in U_0$.
\end{proof}

The following proposition sums up the above:

\begin{prop}\label{prop_nonabelian0}
The chain of subspaces
\[
\RR^{r,s}\supset U_{\Delta}^\bot\supset U_0 \supset \{0\}
\]
is stabilized by $\log(\hol(\Gamma))$ such that
each subspace is mapped to the next in the chain.
\end{prop}

Given the totally isotropic subspace $U_0$, we can find a \emph{Witt basis}
for $\RR^{r,s}$ with respect to $U_0$ as follows:
If $k=\dim U_0$, there exists a basis for $\RR^{r,s}$,
\begin{equation}
\{ u_1,\ldots,u_k,\quad w_1,\ldots,w_{n-2k},\quad u_1^*,\ldots,u_k^* \},
\label{eq_nonabelian2}
\end{equation}
such that $\{u_1,\ldots,u_k\}$ is a basis of $U_0$,
$\{w_1,\ldots,w_{n-2k}\}$ is a basis of a non-degenerate subspace $W$ such that $U_0^\bot = U_0\oplus W$,
and $\{u_1^*,\ldots,u_k^*\}$ is a basis of a space $U_0^*$ such that
$\langle u_i,u_j^*\rangle=\delta_{ij}$ (then $U_0^*$ is called a
\emph{dual space} for $U_0$).
Let $\tilde{I}$
denote the signature matrix representing the restriction of
$\langle\cdot,\cdot\rangle$ to $W$ with respect to
the chosen basis of $W$.

The following generalizes Theorem \ref{thm_wolf2}:

\begin{thm}\label{thm_nonabelian0}
Let $\gamma=(I+A,v)\in\Gamma$ and fix a Witt basis with respect to $U_0$.
Then the matrix representation of $A$ in this basis is
\begin{equation}
A =
\begin{pmatrix}
0 & -B^\top\tilde{I} & C \\
0 & 0 & B \\
0 & 0 & 0
\end{pmatrix},
\label{eq_nonabelian3}
\end{equation}
with $B\in\RR^{(n-2k)\times k}$ and $C\in\sso_k$ (where $k=\dim U_0$).
The columns of $B$ are isotropic and mutually orthogonal with respect to
$\tilde{I}$.
\end{thm}
\begin{proof}
With respect to the given Witt basis, $A$ is represented by a matrix
\[
\begin{pmatrix}
A_1 & -A_2^\top\tilde{I} & A_3 \\
A_4 & A_5 & A_2 \\
A_6 & -A_4^\top\tilde{I} & -A_1^\top
\end{pmatrix}
\]
with $A_3, A_6$ skew-symmetric, $A_5\in\sso(\tilde{I})$.
By Proposition \ref{prop_nonabelian0}, $A_1=0$, $A_4=0$, $A_6=0$, and also
$A_5=0$.
Set $B=A_2$, $C=A_3$.

The condition $A^2=0$
implies $-B^\top\tilde{I}B=0$,
so all columns of $B$ are isotropic and mutually orthogonal with respect
to $\tilde{I}$.
\end{proof}

\section{Dimension Bounds for Non-Abelian Holonomy Groups}
\label{sec_dimbound}

%
We sum up two rules which have to be satisfied by the
representation matrices (\ref{eq_nonabelian3}).
Given matrices $A_i$ ($i=1,2$), $B_i$ and $C_i$ refer
to the respective matrix blocks in (\ref{eq_nonabelian3}).
\begin{enumerate}
\item
\emph{Crossover rule:}
Given $A_1$ and $A_2$,
let $b_2^i$ be a column of $B_2$ and $b_1^k$ a column
of $B_1$. Then
$\langle b_1^k,b_2^i\rangle = - \langle b_1^i,b_2^k\rangle$.
In particular, $\langle b_1^k, b_2^k\rangle=0$, and
$\langle b_1^i,b_1^k\rangle=0$. 
If $\langle b_1^i,b_2^k\rangle\neq 0$ then
$b_1^k,b_1^i,b_2^k,b_2^i$ are linearly independent.
(The product of $A_1 A_2$ contains $-B_1^\top\tilde{I}B_2$ as the skew-symmetric
upper right block, so its entries are the values $-\langle b_1^k,b_2^i\rangle$.)
\item
\emph{Duality rule:}
Assume $A_1$ is not central (that is $A_1A_2\neq0$ for some $A_2$).
Then $B_2$ contains a column $b_2^i$ and $B_1$ a column $b_1^j$ such that
$\langle b_1^j,b_2^i\rangle\neq 0$.
\end{enumerate}

\begin{thm}\label{thm_nonabelian1}
Let $\Gamma\subset\Iso(\RR^{r,s})$ be a group acting
on $\RR^{n}$, $n=r+s$, whose centralizer in $\Iso(\RR^{r,s})$ has an open
orbit.
If $\hol(\Gamma)$ is non-abelian, then
\[
n\geq 8.
\]
As Example \ref{example1} shows, this is a sharp lower bound.
\end{thm}
\begin{proof}
If $\hol(\Gamma)$ is not abelian, there exist
$\gamma_1=(I+A_1,v_1), \gamma_2=(I+A_2,v_2)$ such that $A_1 A_2\neq 0$
(Lemma \ref{lem_wolf2}).

Let $W$ be a vector space complement of $U_0$ in $U_0^\bot$, so
$W$ is non-degenerate and $x\in\RR^{r,s}$ can be written
$x=u+w+u^*$ with $u\in U_0, w\in W, u^*\in U_0^*$.
Then
\[
A_1 x =
\begin{pmatrix}
0 & -B_1^\top\tilde{I} & C_1 \\
0 & 0 & B_1 \\
0 & 0 & 0
\end{pmatrix}\cdot
\begin{pmatrix}
u \\ w\\ u^*
\end{pmatrix}
=\begin{pmatrix}
-B_1^\top\tilde{I}w + C_1 u^* \\
B_1 u^* \\
0
\end{pmatrix}.
\tag{$*$}\label{eq_nonabelian4}
\]
By the duality rule, there are columns in $B_1,B_2$ which are non-orthogonal
to one another. Then, by the crossover rule, $B_1$ and $B_2$ together
contain at least four linearly independent columns. This implies
$\dim W\geq 4$.

Further, $B_1^\top\tilde{I}B_2\neq 0$. So if $A_3=[A_1,A_2]$, this means
the skew-symmetric matrix $C_3$ is non-zero. Hence $C_3$ must have at least
two columns, that is
$\dim U_0 \geq 2$.
Then
\[
n = \dim U_0+\dim W+\dim U_0^* \geq 2+4+2 = 8
\]
holds.
\end{proof}

\begin{remark}\label{rem_nonabelian2}
 With the additional 
assumption that the centralizer of $\Gamma$ in $\Iso(\RR^{r,s})$ acts
transitively, the second author
has a proof (to appear in his dissertation) that the 
dimension bound in Theorem \ref{thm_nonabelian1} can be improved 
to $n\geq 14$. 
As Example \ref{example2} shows, this is a sharp lower bound.
\end{remark}

\section{Examples}
\label{sec_examples}

\begin{lem}\label{lem_example1}
If the centralizer of $\Gamma$ in $\Iso(\RR^{r,s})$ has an open orbit $U$,
then the $\Gamma$-action preserves $U$, that is $\Gamma.U=U$. 
\end{lem}
\begin{proof} By taking the Zariski closure, we may assume
from the beginning that $\Gamma$ is an algebraic subgroup
of $\Iso(\RR^{r,s})$ . Since the elements
of $\Gamma$ are unipotent, the algebraic group $\Gamma$
is also connected. The centralizer $G$ of $\Gamma$ is also
an algebraic subgroup, and as such it has
finitely many open orbits in 
$\RR^{r,s}$ (cf.\ \cite[Proposition 6.8]{Baues}). 
The group $\Gamma$ permutes
the open orbits of $G$. Since it is connected, 
$\Gamma$, in fact,  preserves each orbit.  
\end{proof}

\begin{example}\label{example1}
Let $\Gamma_{4,4}\subset\Iso(\RR^{4,4})$ be the group generated by
\[
\gamma_1 =
\Bigl(
\begin{pmatrix}
I_2 & -B_1^\top\tilde{I} & 0 \\
0 & I_4 & B_1 \\
0 & 0 & I_2
\end{pmatrix},
\begin{pmatrix}
0\\ w_1\\ 0
\end{pmatrix}\Bigr),\quad
\gamma_2 =
\Bigl(
\begin{pmatrix}
I_2 & -B_2^\top\tilde{I} & 0 \\
0 & I_4 & B_2 \\
0 & 0 & I_2
\end{pmatrix},
\begin{pmatrix}
0\\ w_2\\ 0
\end{pmatrix}\Bigr)
\]
in the basis representation (\ref{eq_nonabelian3}). Here,
\[
B_1 =
\begin{pmatrix}
-1 & 0 \\
0 & -1 \\
0 & -1 \\
-1 & 0
\end{pmatrix},\quad
w_1 =
\begin{pmatrix}
1\\0\\0\\1
\end{pmatrix},\quad
B_2 =
\begin{pmatrix}
0 & -1 \\
1 & 0  \\
-1 & 0 \\
0 & 1
\end{pmatrix},\quad
w_2 =
\begin{pmatrix}
0\\-1\\1\\0
\end{pmatrix},
\]
and $\tilde{I}=\left(\begin{smallmatrix} I_2 & 0\\
0&-I_2\end{smallmatrix}\right)$
is the signature matrix of $W$.
Their commutator is
\[
\gamma_3 = [\gamma_1,\gamma_2] =
\Bigl(
\begin{pmatrix}
I_2 & 0 & C_3 \\
0 & I_4 & 0 \\
0 & 0 & I_2
\end{pmatrix},
\begin{pmatrix}
u_3\\ 0\\ 0
\end{pmatrix}\Bigr),
\]
with
\[
C_3 =
\begin{pmatrix}
0 & -4\\
4 & 0
\end{pmatrix},\quad
u_3 =
\begin{pmatrix}
0 \\ -4
\end{pmatrix}.
\]
One checks that $A_i^2=0$ and that $\gamma_{3}$ commutes 
with $\gamma_{1}$, $\gamma_{2}$. Therefore, $\Gamma_{4,4}$ is isomorphic to the discrete Heisenberg group on two generators.

In the chosen basis, the pseudo-scalar product is represented by the matrix
$Q=\left(\begin{smallmatrix}
0&0&I_2\\
0&\tilde{I}&0\\
I_2&0&0
\end{smallmatrix}\right)$.
The following elements $S\in\iso(\RR^{4,4})$, where 
$\iso(\RR^{4,4})$ denotes  the Lie algebra
of $\Iso(\RR^{r,s})$,
commute with
$(A_1,v_1)$ and $(A_2,v_2)$:
\[
S=
\Bigl(
\begin{pmatrix}
S_1 & -S_2^\top\tilde{I} & 0 \\
0 & S_3 & S_2 \\
0 & 0 & -S_1^\top
\end{pmatrix},
\begin{pmatrix}
x\\y\\z
\end{pmatrix}
\Bigr),
\]
where
$x=(x_1,x_2)^\top, y=(y_1,y_2,y_3,y_4)^\top, z=(z_1,z_2)^\top$ are arbitrary and
\[
S_1
=\begin{pmatrix}
z_1 & z_2 \\
z_2 & -z_1
\end{pmatrix},\quad
S_2
=\begin{pmatrix}
-y_1 & y_3-y_2 \\
-y_2 & y_1+y_4 \\
-y_3 & 0 \\
-y_4 & 0
\end{pmatrix},\quad
S_3
=\begin{pmatrix}
0 & 0 & -z_2 & -z_1 \\
0 & 0 & z_1 & -z_2 \\
-z_2 & z_1 & 0 & 0 \\
-z_1 & -z_2 & 0 & 0
\end{pmatrix}.
\]
Hence the elements $\exp(S)$ are contained in the centralizer of $\Gamma_{4,4}$
in $\Iso(\RR^{r,s})$.
As $x,y,z$ are arbitrary, the centralizer of $\Gamma_{4,4}$ has an open
orbit $U$ through the point $0$.
The set of all elements $S$ is not a Lie subalgebra of the centralizer.

\end{example}

\begin{cor}
There exists a flat incomplete homogeneous pseudo-Riemannian
manifold of signature $(4,4)$ with non-abelian linear
holonomy group.
\end{cor}
\begin{proof}
By Lemma \ref{lem_example1} and
Proposition \ref{prop_proper2}, $\Gamma_{4,4}$ 
acts properly
discontinuously and freely on every open orbit $U$ of its centralizer in $\Iso(\RR^{4,4})$.
So $M_{4,4}=U/\Gamma_{4,4}$ is a homogeneous manifold. 
The 
unit vector $e_7$ is a fixed point for $\gamma_3\in\Gamma_{4,4}$, so the action of the centralizer is not transitive.
Hence,  $U \neq \RR^{4,4}$, and $M$ is incomplete.
\end{proof}

\begin{example}\label{example2}
Let $\Gamma_{7,7 }\subset\Iso(\RR^{7,7})$ be the group generated by
\[
\gamma_1 =
\Bigl(
\begin{pmatrix}
I_5 & -B_1^\top\tilde{I} & C_1 \\
0 & I_4 & B_1 \\
0 & 0 & I_5
\end{pmatrix},
\begin{pmatrix}
0\\ 0\\ u_1^*
\end{pmatrix}\Bigr),\quad
\gamma_2 =
\Bigl(
\begin{pmatrix}
I_5 & -B_2^\top\tilde{I} & C_2 \\
0 & I_4 & B_2 \\
0 & 0 & I_5
\end{pmatrix},
\begin{pmatrix}
0\\ 0\\ u_2^*
\end{pmatrix}\Bigr)
\]
in the basis representation (\ref{eq_nonabelian3}). Here,
\[
B_1 =
\begin{pmatrix}
-1 & 0 & 0 & 0 & 0\\
0 & -1 & 0 & 0 & 0\\
0 & -1 & 0 & 0 & 0\\
-1 & 0 & 0 & 0 & 0
\end{pmatrix},\quad
C_1 = \begin{pmatrix}
0 & 0 & 0 & 0 & 0 \\
0 & 0 & 0 & 0 & 0 \\
0 & 0 & 0 & 0 & -1 \\
0 & 0 & 0 & 0 & 0 \\
0 & 0 & 1 & 0 & 0
\end{pmatrix},\quad
u_1^* =
\begin{pmatrix}
0\\0\\0\\-1\\0
\end{pmatrix},
\]
\[
B_2 =
\begin{pmatrix}
0 & -1 & 0 & 0 & 0\\
1 & 0 & 0 & 0 & 0\\
-1 & 0 & 0 & 0 & 0\\
0 & 1 & 0 & 0 & 0
\end{pmatrix},\quad
C_2 = \begin{pmatrix}
0 & 0 & 0 & 0 & 0 \\
0 & 0 & 0 & 0 & 0 \\
0 & 0 & 0 & 0 & 0 \\
0 & 0 & 0 & 0 & -1 \\
0 & 0 & 0 & 1 & 0
\end{pmatrix},\quad
u_2^* =
\begin{pmatrix}
0\\0\\1\\0\\0
\end{pmatrix},
\]
and $\tilde{I}=\left(\begin{smallmatrix} I_2 & 0\\
0&-I_2\end{smallmatrix}\right)$
is the signature matrix of $W$.
Their commutator is
\[
\gamma_3 = [\gamma_1,\gamma_2] =
\Bigl(
\begin{pmatrix}
I_5 & 0 & C_3 \\
0 & I_4 & 0 \\
0 & 0 & I_5
\end{pmatrix},
\begin{pmatrix}
u_3\\ 0\\ 0
\end{pmatrix}\Bigr),
\]
with
\[
C_3 =
\begin{pmatrix}
0 & -4 & 0 & 0 & 0 \\
4 & 0 & 0 & 0 & 0 \\
0 & 0 & 0 & 0 & 0 \\
0 & 0 & 0 & 0 & 0 \\
0 & 0 & 0 & 0 & 0
\end{pmatrix},\quad
u_3 =
\begin{pmatrix}
0\\0\\0\\0\\2
\end{pmatrix}.
\]
One checks that $A_i^2=0$ and that $\Gamma_{7,7}$ is isomorphic to a
discrete Heisenberg group.

In the chosen basis, the pseudo-scalar product is represented by the matrix
$Q=\left(\begin{smallmatrix}
0&0&I_5\\
0&\tilde{I}&0\\
I_5&0&0
\end{smallmatrix}\right)$.
The following elements $S\in\iso(\RR^{7,7})$ commute with
$(A_1,v_1)$ and $(A_2,v_2)$:
\[
S=
\Bigl(
\begin{pmatrix}
S_1 & -S_2^\top\tilde{I} & S_3 \\
0 & 0 & S_2 \\
0 & 0 & -S_1^\top
\end{pmatrix},
\begin{pmatrix}
x\\y\\z
\end{pmatrix}
\Bigr),
\]
where
$x=(x_1,\ldots,x_5)^\top, y=(y_1,\ldots,y_4)^\top, z=(z_1,\ldots,z_5)^\top$ are arbitrary and
\[
S_1
=\begin{pmatrix}
0 & 0 & 0 & 0 & -2 z_2 \\
0 & 0 & 0 & 0 & 2 z_1 \\
0 & 0 & 0 & 0 & 0 \\
0 & 0 & 0 & 0 & 0 \\
0 & 0 & 0 & 0 & 0
\end{pmatrix},\quad
S_2
=\begin{pmatrix}
0 & 0 &-z_2 & z_1 & 0 \\
0 & 0 & z_1 & z_2 & 0 \\
0 & 0 & -z_1 & z_2 & 0 \\
0 & 0 & z_2 & z_1 & 0
\end{pmatrix},
\]
\[
S_3
=\begin{pmatrix}
0 & 0 & -y_2-y_3 & y_4-y_1 & 0 \\
0 & 0 & y_1+y_4 & y_3-y_2 & 0 \\
y_2+y_3 & -y_1-y_4 & 0 & z_5 & -z_4 \\
y_1-y_4 & y_2-y_3 & -z_5 & 0 & z_3 \\
0 & 0 & z_4 & -z_3 & 0
\end{pmatrix}.
\]
The linear part of such a matrix $S$
is conjugate to a strictly upper triangular matrix
via conjugation with the matrix
\[
T=(e_1,\ e_2,\ e_3,\ e_4,\ e_7+e_8,\ e_5,\ e_6,\ e_9,\ e_{10},\ e_{11},\ e_{12},\ e_{13},\ e_{14},\ e_7-e_8),
\]
where $e_i$ denotes the $i$th unit vector.
Hence, the elements $\exp(S)$ generate a uni\-potent group of isometries whose translation
parts contain all of $\RR^{14}$. Therefore,  the centralizer of $\Gamma_{7,7}$
in $\Iso(\RR^{7,7})$ acts transitively 
(see \cite[Corollary 6.27]{Baues}
or \cite[Theorem 4.2]{DuIh_2}).
In particular, $\Gamma_{7,7}$ acts freely on $\RR^{7,7}$.

It can be verified that the set of all matrices $S$ forms a 3-step nilpotent Lie subalgebra of the
centralizer algebra. Hence the set of all $\exp(S)$ forms a unipotent group
of isometries acting simply transitively on $\RR^{7,7}$.
\end{example}


\begin{cor}
There exists a flat complete homogeneous pseudo-Riemannian
mani\-fold of signature $(7,7)$ with non-abelian linear
holonomy group.
\end{cor}
\begin{proof}
By Proposition \ref{prop_proper1}, the group $\Gamma_{7,7}$ 
acts properly discontinuously and freely on $\RR^{7,7}$.
So $M=\RR^{7,7}/\, \Gamma_{7,7}$ is a complete 
homogeneous manifold.
\end{proof}

\section{Properness of Actions with Transitive Centralizer}
\label{sec_proper}

Recall that an action of a Lie group $L$ on a locally compact
Hausdorff space $X$ is called \emph{proper} if and only if for all
compact sets $K \subset X$ the set 
$\{ \ell  \in L \mid \ell K \cap K \neq \emptyset
\}$ is compact.  

\begin{lem} Let $X= G/H$ be a homogeneous space, where
$G$ is a Lie group and $H$ is a closed subgroup.
Let $L \leq \Diff(X)$ be a group of diffeomorphisms of $X$
which centralizes $G$. 
Then $L$ acts properly on $X$ if and only if $L$ is 
a closed subgroup of $\Diff(X)$ with respect to the 
compact open topology. 
\end{lem}
\begin{proof} 
Choose a basepoint $x_{0} \in X$ such that $H= G_{x_{0}}$ is the stabilizer of $x_{0}$. Then $X$ is homeomorphic to $G/H$ 
via the orbit map $o: G/H \ra X$, $g \mapsto g \cdot x_{0}$. 
The right-action of $\mathrm{N}_{G}(H)$ on $G$ induces a 
continuous homomorphism
onto the centralizer $\mathrm{Z}_{X}(G)$ 
of $G$ in $\Diff(X)$.
Let $\overline L$ denote the preimage
of $L$ in $\mathrm{N}_{G}(H)$. In particular, if $L$ is closed in
$\Diff(X)$ then $\overline L$ is closed in $G$. 
Note that $ X /  L = G / \overline L$ is a Hausdorff space if and only
if the subgroup $\overline L$ is closed in $G$. Since $L$ acts freely
on $X$, $X/ L$ is Hausdorff if and only if $L$ acts properly
on $X$. This proves the lemma.
\end{proof}

We can apply this criterion in the affine situation, as follows: 

\begin{prop} \label{prop_proper1}
Let $L \leq \Aff(\RR^n)$ be a subgroup whose
centralizer in $\Aff(\RR^n)$ acts transitively on  $\RR^n$. 
Then the action of $L$ on $\RR^n$ is proper
if and only if\/ $L$ is a closed subgroup of $\Aff(\RR^n)$.
\end{prop} 

Similarly, assume that the centralizer $G$ of $L$ in $\Aff(\RR^n)$
has an open orbit $U = G \cdot x_{0}$ which is preserved by $L$.
Then
$L$ acts freely on $U$, and the action is proper if and
only if $L$ is closed in $\Diff(U)$. Since $\Diff(U) 
\cap \Aff(\RR^n)$ 
is closed in $\Aff(\RR^n)$ (cf.\ \cite[Lemma 6.9]{Baues}), 
the above proposition generalizes to: 

\begin{prop} \label{prop_proper2}
Let $L \leq \Aff(\RR^n)$ be a subgroup whose
centralizer in $\Aff(\RR^n)$ acts transitively on  an open subset
$U$ of\/ $\RR^n$. 
If $L$ preserves $U$ then the action of $L$ on $U$ is proper
if and only if\/ $L$ is a closed subgroup of $\Aff(\RR^n)$.
\end{prop} 

\begin{remark} P\"uttmann \cite[Section 4.2]{Puettmann} gives an example
of a free
action of the abelian group $(\CC^2,+)$ on $\CC^5$ by unipotent affine 
transformations, such that the quotient is not a Hausdorff
space. Hence the action is not proper.
\end{remark}


\begin{thebibliography}{10}

\bibitem{Baues} O.\ Baues, {\em Flat Pseudo-Riemannian manifolds and prehomogeneous affine representations}, 
in 'Handbook of pseudo-Riemannian Geometry and Supersymmetry', EMS, IRMA Lectures in Mathematics and Theoretical Physics {\bf 16}, 2010, pp. 731-817.
(also arXiv:0809.0824v1).

\bibitem{DuIh} D.\ Duncan, E.\ Ihrig, 
{\em Incomplete flat homogeneous geometries},
Differential geometry: geometry in mathematical physics and related topics (Los Angeles, CA, 1990), 197-202, Proc. Sympos. Pure Math., 54, Part 2, Amer. Math. Soc., Providence, RI, 1993. 


\bibitem{DuIh_2}  D.\ Duncan, E.\ Ihrig, 
{\em Flat pseudo-Riemannian manifolds with a nilpotent transitive group of isometries}, Ann.\ Global Anal.\ Geom.\ {\bf 10} (1992), no. 1, 87-101. 

\bibitem{DuIh_1}  D.\ Duncan, E.\ Ihrig, {\em Homogeneous spacetimes of zero curvature},
Proc.\  Amer.\ Math.\ Soc.\ {\bf 107}  (1989), no. 3, 785-795.


\bibitem{Marsden} J.\ Marsden, {\em 
On completeness of homogeneous pseudo-riemannian manifolds},
Indiana Univ.\ J.\ {\bf 22} (1972/73), 1065-1066.

\bibitem{Puettmann} A. P\"uttmann, {\em Free affine actions of unipotent groups on $\CC^n$},
Transformation Groups {\bf 12} (2007), no. 1, 137-151.

\bibitem{Wolf_buch} J.A.\ Wolf,
{\em Spaces of constant curvature}, 6th edition,
AMS, to appear 

\bibitem{Wolf_1} J.A.\ Wolf, {\em Homogeneous manifolds of zero curvature},
Trans. Amer. Math. Soc. {\bf 104} (1962), 462-469.

%
%

\bibitem{Wolf_2} J.A.\ Wolf,
{\em  Flat homogeneous pseudo-Riemannian manifolds}, Geom. Dedicata {\bf 57} (1995), no. 1, 111-120.

\bibitem{Wolf_3} J.A.\ Wolf,
{\em Isoclinic spheres and flat homogeneous pseudo-Riemannian manifolds},  Crystallographic groups and their generalizations (Kortrijk, 1999), 303-310, Contemp. Math., {\bf 262}, Amer.\ Math.\ Soc., Providence, RI, 2000.

\end{thebibliography}
\end{document}